\documentclass[11pt]{article}
\usepackage{graphicx}
\usepackage{graphics}
\usepackage{amsfonts}
\usepackage{amscd}
\usepackage{amsthm}
\usepackage{indentfirst}
\usepackage[all]{xy}
\usepackage{amssymb, amsmath, amsthm, amsgen, amstext, amsbsy, amsopn}
\usepackage{epic,eepic}
\usepackage{hyperref}
\usepackage{enumitem}
\usepackage{color}
\usepackage{dsfont}
\usepackage[margin=1.3in]{geometry}
\usepackage{comment}

\usepackage{xcolor}

\usepackage{appendix}

\theoremstyle{plain}
\newtheorem{thm}{Theorem}[section]  
\newtheorem{lemma}[thm]{Lemma}  
\newtheorem{prop}[thm]{Proposition}  
\newtheorem{claim}[thm]{Claim}  

\theoremstyle{definition}
\newtheorem{defn}[thm]{Definition}  
\newtheorem{rem}[thm]{Remark}  

\theoremstyle{definition}
\newtheorem{quest}[thm]{Question} 

\title{\textbf{$C^0$-flexibility of Legendrian discs}}
\author{\textbf{Maksim Stoki\'c}}
\date{\empty}

\begin{document}

\maketitle

\begin{abstract}
    \noindent We construct a compactly supported contact homeomorphism of $\mathbb{R}^5$ that maps a Legendrian disc to a smooth disc that is nowhere Legendrian. This flexibility result comes in contrast with the case of closed Legendrians, where Dimitroglou-Rizell and Sullivan proved that $C^0$-rigidity holds.

    \bigskip
    \noindent\textbf{Note added:} We noticed a gap in the proof of Proposition \ref{PropZigZagStrips}. We are currently working on fixing it.
\end{abstract}

\section{Introduction}

\noindent A \textit{contact homeomorphism} of a contact manifold $(Y, \xi)$ is defined as a homeomorphism that arises as the $C^0$-limit of a sequence of contact diffeomorphisms. The contact version of the celebrated Eliashberg-Gromov rigidity theorem implies that a smooth contact homeomorphism preserves the contact structure. In this paper, we focus on Legendrian submanifolds and try to understand their images under contact homeomorphisms. More precisely, our aim is to answer the following question.

\begin{quest}\label{Question1}
    Let $\Lambda\subset (Y,\xi)$ be a Legendrian submanifold, and let $\phi:Y\rightarrow Y$ be a contact homeomorphism such that the image $\phi(\Lambda)$ is a smooth submanifold of $Y$. Must $\phi(\Lambda)$ be Legendrian?
\end{quest}

\noindent The above question has been extensively studied, and we summarize some of the relevant results. Rosen and Zhang (see \cite{RZ20}) proved the $C^0$-rigidity of Legendrians under the assumption that the approximating sequence of contact diffeomorphisms has conformal factors that converge uniformly. Usher (see \cite{Us20}) further relaxed this assumption by requiring only a uniform lower bound on the sequence of conformal factors. The first result without any assumptions on the conformal factors was obtained in dimension 3 by Dimitroglou Rizell and Sullivan, who proved (see \cite{RS22a}) that a contact homeomorphism cannot map a Legendrian knot to a smooth non-Legendrian knot. Shortly after, another proof of the $C^0$-rigidity of Legendrian knots appeared in \cite{St22}, which inspired Dimitroglou Rizell and Sullivan to establish the $C^0$-rigidity of closed Legendrians in any dimension (see \cite{RS22b}). Contrary to all previous results, we provide a negative answer to the above question when $\Lambda$ is a smoothly embedded Legendrian disc in $(\mathbb{R}^5, \ker(dz - ydx - pdq))$.

\begin{thm}\label{TheoremMain}
    There exists a contact homeomorphism $\Phi$ of $\mathbb{R}^5$ with the standard contact form, which maps smoothly embedded Legendrian disc $\Lambda:[-1,1]^2\rightarrow\mathbb{R}^5$ to a smoothly embedded nowhere Legendrian disc $\Lambda_{\infty}:[-1,1]^2\rightarrow\mathbb{R}^5,\,(t,s)\mapsto(0,0,t,s,0)$.
\end{thm}

\begin{rem}
    As mentioned above, Dimitroglou-Rizell and Sullivan proved the $C^0$-rigidity of closed Legendrians, which, together with Theorem \ref{TheoremMain}, implies that $C^0$-Legendrian rigidity is a global phenomenon. To the best of the author’s knowledge, this is the first such phenomenon in $C^0$-contact and symplectic geometry. For example, the $C^0$-rigidity of Lagrangian submanifolds (see \cite{HLS15}) applies to both open and closed Lagrangian submanifolds.
\end{rem}

\begin{rem}
    We believe that the same construction works in $\mathbb{R}^{2n+1} \cong \mathbb{R}^3 \times T^*\mathbb{R}^{n-1}$ for $n > 1$. In this case, we can map a Legendrian $n$-disc to a non-Legendrian one, defined as the product of a transverse curve in $\mathbb{R}^3$ and the $(n-1)$-disc in the zero-section of $T^*\mathbb{R}^{n-1}$. The three-dimensional case remains open, and our method relies on the $h$-principle for loose Legendrians, which only applies in dimensions $\geq 5$.
\end{rem}

\noindent Recall that a contact homeomorphism $\Phi$ comes with a sequence of contactomorphisms $\{\Phi_i\}_{i\geq 1}$ such that $\Phi_i\xrightarrow[i\rightarrow\infty]{C^0}\Phi$. Let $D\subset\mathbb{R}^5$ be a Legendrian disc such that the image $\Phi(D)$ is smooth. Consider the following two cases:
\begin{enumerate}
    \item The sequence \(\{\Phi_i\}_{i \geq 1}\) is fixed near the boundary \(\partial D\).
    \item No restrictions are imposed on the sequence \(\{\Phi_i\}_{i \geq 1}\) near the boundary \(\partial D\).
\end{enumerate}

\noindent In the first case, the $C^0$-Legendrian rigidity for closed manifolds implies that the image $\Phi(D)$ must be Legendrian. In contrast, in the second case, Theorem \ref{TheoremMain} implies that the image $\Phi(D)$ can be non-Legendrian. The next important step is to identify the transition point between rigidity and flexibility.

\begin{quest}\label{Question2}
    What is the natural boundary condition (on the sequence $\{\Phi_i\}_{i\geq 1}$ near $\partial D$) under which the image $\Phi(D)$ has to be Legendrian?
\end{quest}

\noindent A more general version of Question \ref{Question1}, when $\Lambda$ is coisotropic (instead of Legendrian), remains open. We believe that Question \ref{Question2} provides a good starting point for tackling the $C^0$-coisotopic rigidity.

\subsection{Acknowledgments}

\noindent I thank Lev Buhovsky, Yakov Eliashberg, and Leonid Polterovich for the valuable discussions and interest in this work. The author was partially supported by ERC Starting Grant 757585 and ISF Grant 0603603171.

\section{Proof of Theorem \ref{TheoremMain}}\label{Section3}

\noindent\textbf{Step I)} \textit{We construct a sequence of Legendrian embeddings $\big\{\Lambda_i:[-1,1]^2\rightarrow\mathbb{R}^5\big\}_{i=1}^{\infty}$, and a non-Legendrian embedding $\Lambda_{\infty}:[-1,1]^2\to\mathbb{R}^5$ such that:}
\begin{equation}\label{Condition1}
    \sum_{i=1}^{\infty}d_{C^0}(\Lambda_i,\Lambda_{\infty})<\infty.
\end{equation}

\noindent\textbf{Step II)} \textit{We construct a sequence of contactomorphisms $\big\{\Psi_i:\mathbb{R}^5\rightarrow\mathbb{R}^5\big\}_{i=1}^{\infty}$, supported in a single compact subset of $\mathbb{R}^5$, such that:}
\begin{equation}\label{Condition2}
    \Psi_i\circ\Lambda_i=\Lambda_{i+1},\quad ||\Psi_i||_{C^0}\leq C/2^i
\end{equation}
\begin{equation}\label{Condition3}
    (\forall\, U\subset\overline{U}\subset\mathbb{R}^5\setminus\mathrm{Im}\,\Lambda_1)(\exists\,m=m(U))\quad \Phi_m|_U=\Phi_{m+1}|_{U}=\Phi_{m+2}|_{U}=\cdots
\end{equation}
\noindent\textit{where $\Phi_i:=\Psi_i\circ\Psi_{i-1}\circ\ldots\circ\Psi_1$.}

\begin{claim}\label{ClaimConditions(1,2,3)implyThm1}
    If the sequences $\{\Lambda_i\}_{i=1}^{\infty}$ and $\{\Psi_i\}_{i=1}^{\infty}$ satisfy properties (\ref{Condition1}),(\ref{Condition2}), (\ref{Condition3}), then the sequence $\Psi_{i}$ $C^0$-converges to a homeomorphism $\Phi$ that satisfies $\Phi\circ\Lambda_1=\Lambda_{\infty}$.
\end{claim}

\begin{proof}
    Property (\ref{Condition2}) implies that $\{\Phi_i\}_{i=1}^{\infty}$ is a Cauchy sequence, and therefore it converges uniformly to a continuous map $\Phi$. Moreover, property (\ref{Condition1}) implies that $\Phi\circ\Lambda_1=\Lambda_{\infty}$. Let us prove that $\Phi$ is an injective map, and hence a homeomorphism. Let $x\neq y$ be different points in $\mathbb{R}^5$. If $x,y\in\mathrm{Im}\,\Lambda_1$ then $\Phi(x)=\Lambda_{\infty}\circ\Lambda_1^{-1}(x)\neq \Lambda_{\infty}\circ\Lambda_1^{-1}(y)=\Phi(y)$. If $x,y\in\mathbb{R}^5\setminus\mathrm{Im}\,\Lambda_1$, then the property (\ref{Condition3}) implies that for $m$ large enough we have $\Phi(x)=\Phi_m(x)\neq\Phi_m(y)=\Phi(y)$. Finally, assume $x\in\mathrm{Im}\,\Lambda_1$, $y\in\mathbb{R}^5\setminus\mathrm{Im}\,\Lambda_1$, and let $B\subset\mathbb{R}^5$ be an open ball centered in $y$, such that $\overline{B}\cap\mathrm{Im}\,\Lambda_1=\emptyset$. The property (\ref{Condition3}) implies that for $m$ large enough we have $\Phi_m|_{B}=\Phi|_{B}$, and hence $\Phi_m(\mathrm{Im}\,\Lambda)\in\mathbb{R}^5\setminus\Phi_m(B)=\mathbb{R}^5\setminus\Phi(B)$. Passing to the limit $m\rightarrow\infty$ yields $\Phi(x)\in\Phi(\mathrm{Im}\,\Lambda_1)\subset\mathbb{R}^5\setminus\Phi(B)$, and so $\Phi(x)\neq\Phi(y)$.
\end{proof}

\noindent Let us prove using induction that the sequences with properties (\ref{Condition1}),(\ref{Condition2}), and (\ref{Condition3}) do exist. We will strengthen these conditions in order to be able to complete the induction.\\

\noindent\textbf{Strengthening condition (\ref{Condition1}):} Let $\gamma_m:[-1,1]\rightarrow(\mathbb{R}^3,\mathrm{ker}\,dz-ydx)$ be a sequence of Legendrian embeddings defined as
\begin{equation}\label{Definition_gamma_m(t)}
    \gamma_m(t)=\Big(\frac{\sin{m^2t}-\cos{m^2t}}{m},\frac{\sin{m^2t}+\cos{m^2t}}{m},t-\frac{\cos{2m^2t}}{2m^2}\Big)
\end{equation}
\noindent We require that the sequence $\{\Lambda_i\}_{i=1}^{\infty}$ has form
\begin{equation}\label{ConditionLambda}
    \Lambda_i(t,s):=(\gamma_{m_i}(t),s,0)
\end{equation}
\noindent where $\{m_i\}_{i=1}^{\infty}$ is a sequence of natural numbers that satisfies
\begin{equation}\label{ConditionOn(m_i)}
    \forall i\in\mathbb{N}\quad m_{i+1}>\max\{2^i,m_i\}
\end{equation}
\noindent Note that $\gamma_m$ $C^0$-converges to a transverse curve $\gamma_{\infty}(t):=(0,0,t)$. The condition (\ref{ConditionOn(m_i)}) implies that the series $\sum_{i=1}^{\infty}1/m_i$ converges, and therefore the sequence $\{\Lambda_i\}_{i=1}^{\infty}$ satisfies condition (\ref{Condition1}).\\

\noindent\textbf{Strengthening condition (\ref{Condition3}):} To each Legendrian $\Lambda_i(t,s)=(\gamma_{m_i}(t),s,0)$ we assign an open subset $\mathcal{V}_i\subset\mathbb{R}^5$ defined as
\begin{equation}\label{EquationV_k}
    \mathcal{V}_i=\Big\{|x|<\frac{10}{m_i},|y|<\frac{10}{m_i},|z|<1+\frac{10}{m_i},|q|<1+\frac{10}{m_i},|p|<\frac{10}{m_i}\Big\}\supset\mathrm{Im}\,\Lambda_i\cup \mathrm{Im\,\Lambda_{\infty}}
\end{equation}
\noindent Let $\{\mathcal{U}_i\}_{i=1}^{\infty}$ be a decreasing sequence of open sets with $\mathcal{U}_1=\mathbb{R}^5$ and $\bigcap_{i\geq 1}\mathcal{U}_i=\mathrm{Im}\,\Lambda_1$. Finally, we require that a sequence of contactomorphisms $\{\Psi_i\}_{i=1}^{\infty}$ satisfies (\ref{Condition2}) and
\begin{equation}\label{Condition3upgraded}
    \mathrm{supp}\,\Psi_i\subset\mathcal{V}_i\subset\Phi_{i-1}(\mathcal{U}_{i}),\quad\text{ where }\Phi_0:=\mathrm{Id}\text{ and }\Phi_i=\Psi_{i}\circ\Psi_{i-1}\circ\cdots\circ\Psi_1
\end{equation}

\begin{claim}\label{Claim(3)implies(3)upgraded}
    The condition (\ref{Condition3upgraded}) implies condition (\ref{Condition3}). Additionally, it implies that all contactomorphisms $\Psi_i$ have support in a compact set $\overline{\mathcal{V}_1}$.
\end{claim}

\begin{proof}
    Let $U\subset\overline{U}\subset\mathbb{R}^5\setminus\mathrm{Im}\,\Lambda_1$ and pick $m_0$ large enough so that $\overline{U}\subset\mathbb{R}^5\setminus\mathcal{U}_{m_0}$ and consequently $\Phi_{m_0}(\overline{U})\subset\mathbb{R}^5\setminus\Phi_{m_0}(\mathcal{U}_{m_0})$. Note that the property (\ref{Condition3upgraded}) implies that for $m>m_0$ we have $\Psi_m|_{\mathbb{R}^5\setminus\Phi_{m_0}(\mathcal{U}_{m_0})}=\mathrm{Id}$, and therefore $\Phi_m(p)=\Psi_m\circ\cdots\circ\Psi_{{m_0}+1}\circ\Phi_{m_0}(p)=\Phi_{m_0}(p)$ for every $p\in\overline{U}$.
\end{proof}

\noindent Using the induction we will prove that the sequences $\{m_i\}_{i=1}^{\infty}$, $\{\Lambda_i\}_{i=1}^{\infty}$ and $\{\Psi_i\}_{i=1}^{\infty}$ that satisfy properties (\ref{ConditionLambda}),(\ref{ConditionOn(m_i)}),(\ref{Condition2}) and (\ref{Condition3upgraded}) do exist. Then the Theorem \ref{TheoremMain} follows from Claims \ref{ClaimConditions(1,2,3)implyThm1} and \ref{Claim(3)implies(3)upgraded}.\\

\noindent Assume we have constructed natural numbers $m_1<m_2<\ldots<m_k$, Legendrian embeddings $\Lambda_1,\ldots,\Lambda_k$ and contact diffeomorphisms $\Psi_1,\ldots,\Psi_{k-1}$ that satisfy all the properties listed above. We need to construct $m_{k+1}>2\cdot m_k$ and a contact diffeomorphism $\Psi_k$ supported inside $\mathcal{V}_{k}$ such that $\mathcal{V}_{k+1}\subset\Psi_k\circ\Phi_{k-1}(\mathcal{U}_{k+1})$, $\Psi_k\circ\Lambda_{k}=\Lambda_{k+1}$ and $d_{C^0}(\Psi_k,\mathrm{Id})<C/2^k$.

\subsection*{Step 1:}

\noindent\textit{In this step, we want to construct a $C^0$-small contactomorphism $\Psi_k'$ that "stretches" the open set $\Phi_{k-1}(\mathcal{U}_{k+1})\supset\mathrm{Im}\,\Lambda_k$ all the way to the $\mathrm{Im}\,\Lambda_{\infty}$. Moreover, we want to find an embedding $\Pi:[-1,1]^2\times[0,1]\hookrightarrow\mathbb{R}^5$, with $\Pi(t,s,0)=\Lambda_k(t,s)$ and $\Pi(t,s,1)=\Lambda_{\infty}(t,s)$ so that $\mathrm{Im}\,\Pi\subset\Psi'_k\circ\Phi_{k-1}(\mathcal{U}_{k+1})$, and it will serve as a guideline for constructing isotopies between $\Lambda_{k}$ and $\Lambda_{k+1}$. First we focus on the contact manifold $\mathbb{R}^3$, where we find an embedded surface $\Sigma:[-1,1]\times[0,1]\rightarrow\mathbb{R}^2$ with $\Sigma(\cdot,0)=\gamma_{m_k},\,\Sigma(\cdot,1)=\gamma_{\infty}$, and then we extend results to the product $\mathbb{R}^5=\mathbb{R}^3\times\mathbb{R}^2$ where we define $\mathrm{\Pi}(t,s,\tau)=(\Sigma(t,\tau),s,0)$.}\\

\noindent Let $\Sigma:[-1,1]\times[0,1]\rightarrow(\mathbb{R}^3,\mathrm{ker}\,\alpha)$, $\alpha=dz-ydx$ be an embedded surface defined as
\begin{equation*}
    \Sigma(t,w)=\Big(\frac{1-w}{m_k}(\sin{m_k^2t}-\cos{m_k^2t}),\frac{1-w}{m_k}(\sin{m_k^2t}+\cos{m_k^2t}),t-\frac{(1-w)^2}{2m_k^2}\cos{2m_k^2t}\Big)
\end{equation*}

\noindent Additionally, define an embedding $\Pi:[-1,1]^2\times[0,1]\rightarrow\mathbb{R}^5$
\[\Pi(t,s,w)=(\Sigma(t,w),s,0),\]
\noindent and note that for each $(t,s)\in[-1,1]^2$ we have $\Pi(t,s,0)=\Lambda_k(t,s)$, $\Pi(t,s,1)=\Lambda_{\infty}(t,s)$ and $\mathrm{diam}\,\Pi(\{(t,s)\}\times[0,1])<\frac{c}{2^k}$.

\begin{claim}
    There exists a contact diffeomorphism $\Psi'_k$ supported inside $\mathcal{V}_k$ such that $d_{C^0}(\Psi'_k,\mathrm{Id})<C/2^k,\,\Psi'_k\circ\Lambda_k=\Lambda_k$ and $\mathrm{Im}\,\Pi\subset\Psi'_k\circ\Phi_{k-1}(\mathcal{U}_{k+1})$.
\end{claim}

\begin{proof}
    \noindent Since $\Pi(t,s,0)=\Lambda_k(t,s)$, there exists $w_0\in (0,1)$ such that
    \[\Pi([-1,1]^2\times[0,w_0])\subset\Phi_{k-1}(\mathcal{U}_{k+1}).\]
    
    \noindent Let $\Gamma_{\tau}:[-1,1]^2\times[0,1]\rightarrow\mathbb{R}^5,\tau\in[0,1]$ be an isotopy of embeddings defined as
    \[\Gamma_{\tau}(t,s,w)=\Pi\big(t,s,w+\lambda_{\tau}(t,s)(w/w_0-w)\big)=(\Sigma(t,w+\lambda_{\tau}(t,s)(w/w_0-w)),s,0),\]
    \noindent where $\lambda_{\tau}$ is the function from Lemma~\ref{lambda_function_lemma} for some $\varepsilon>0$ and sets $\mathcal{A}_i$. The isotopy $\Gamma_{\tau}$ maps $\Pi([-1,1]^2\times[0,w_0])$ to $\Pi([-1,1]^2\times[0,1])$ by flowing along the lines $w\mapsto\Pi(\cdot,\cdot,w)$. Its support satisfies $\mathrm{supp}\,\Gamma_{\tau} \subset Op(\Pi((\mathrm{supp}\,\lambda_{\tau})\times[0,1]))$. By the properties of $\lambda_{\tau}$, for each $i\in\{0,1,2,3\}$ and $\tau\in[i/4,(i+1)/4]$ we have
    \[
        \mathrm{supp}_{\tau\in[i/4,(i+1)/4]}\,\Gamma_{\tau} \subset \mathcal{B}_i,
    \]
    \noindent where $\mathcal{B}_i = Op(\Pi(\mathcal{A}_i\times[0,1])) \subset \mathbb{R}^5$ is a disjoint union of sets with diameter $\leq c/2^k$ when $\varepsilon>0$ is sufficiently small.
    
    \noindent The final step consists of extending $\Gamma_{\tau}$ to an ambient contact isotopy $\phi_{\tau}$ generated by a contact Hamiltonian $H_{\tau}$ - once this is achieved, the construction will be complete. Indeed, after applying a cut-off, we may assume that for each $i\in\{0,1,2,3\}$ and $\tau\in[i/4,(i+1)/4]$, both $H_{\tau}$ and $\phi_{\tau}$ are supported in $\mathcal{B}_i$. The triangle inequality then yields $d_{C^0}(\phi_{1},\mathrm{Id}) < 4c/2^k$, and defining $C=4c$ with $\Psi'_k=\phi_1$ finishes the argument.\\
    
    \noindent Let us show that there exists a contact Hamiltonian function $H_{\tau}\in C^{\infty}_c(\mathbb{R}^5)$ such that $X_{H_{\tau}}|_{\Gamma_{\tau}}=\frac{\partial}{\partial\tau}\Gamma_{\tau}$. The vector field $\frac{\partial}{\partial\tau}\Gamma_{\tau}$ is given by
    \begin{align*}
        \frac{\partial}{\partial\tau}\Gamma_{\tau} =\; & \frac{\partial \lambda_{\tau}}{\partial \tau} \left( \frac{w}{w_0} - w \right) \cdot \Bigg(
        \frac{\cos m_k^2 t - \sin m_k^2 t}{m_k} \cdot \frac{\partial}{\partial x}
        - \frac{\cos m_k^2 t + \sin m_k^2 t}{m_k} \cdot \frac{\partial}{\partial y} \\
        & + \frac{1 - \left( w + \lambda_{\tau}\cdot\left( \frac{w}{w_0} - w \right)\right)}{m_k^2} \cos(2m_k^2 t) \cdot \frac{\partial}{\partial z}
    \Bigg).
    \end{align*}

    \noindent Let $\beta=dz-ydx-pdq$ be the standard contact form on $\mathbb{R}^5$. Recall that the contact Hamiltonian vector field $X_H$ satisfies equations $\beta(X_{H_{\tau}})=H_{\tau}$ and $i_{X_{H}}d\beta=dH_{\tau}(R_{\beta})\beta-dH_{\tau}$, where $R_{\beta}=\frac{\partial}{\partial z}$ denotes the Reeb vector field associated with $\beta$. In coordinates one can write $X_{H_{\tau}}$ as
    \begin{align*}
        X_{H_{\tau}} =\; & -\frac{\partial H_{\tau}}{\partial y} \frac{\partial}{\partial x}
        + \left( y \frac{\partial H_{\tau}}{\partial z} + \frac{\partial H_{\tau}}{\partial x} \right) \frac{\partial}{\partial y}
        + \left( H_{\tau} - y \frac{\partial H_{\tau}}{\partial y} - p \frac{\partial H_{\tau}}{\partial p} \right) \frac{\partial}{\partial z} \\
        & - \frac{\partial H_{\tau}}{\partial p} \frac{\partial}{\partial q}
        + \left( p \frac{\partial H_{\tau}}{\partial z} + \frac{\partial H_{\tau}}{\partial q} \right) \frac{\partial}{\partial p}.
    \end{align*}

    \noindent Since $\beta\left(\frac{\partial}{\partial\tau}\Gamma_{\tau}\right)=0$, it suffices to ensure that the function $H_{\tau}$ satisfies along $\Gamma_{\tau}$:
    \begin{enumerate}[label=(\roman*)]
        \item $H_{\tau}\circ\Gamma_{\tau}\equiv 0$,
        \item $\displaystyle dx\left(\frac{\partial \Gamma_{\tau}}{\partial \tau}\right)=\frac{\partial \lambda_{\tau}}{\partial \tau}\left(\frac{w}{w_0} - w\right) \cdot 
    \frac{\cos m_k^2 t - \sin m_k^2 t}{m_k}=-\frac{\partial H_{\tau}}{\partial y}$,
        \item $\displaystyle dy\left(\frac{\partial \Gamma_{\tau}}{\partial \tau}\right)=\frac{\partial \lambda_{\tau}}{\partial \tau}\left(\frac{w}{w_0} - w\right) \cdot 
    \frac{\cos m_k^2 t + \sin m_k^2 t}{m_k}=-y\frac{\partial H_{\tau}}{\partial z}-\frac{\partial H_{\tau}}{\partial x}$,
        \item $\displaystyle dz\left(\frac{\partial \Gamma_{\tau}}{\partial \tau}\right)=\frac{\partial \lambda_{\tau}}{\partial \tau}\left(\frac{w}{w_0} - w\right) \cdot \frac{1-\left(w+\lambda_{\tau}\left(\frac{w}{w_0}-w\right)\right)}{m_k^2}\cos 2m_k^2t=-y\frac{\partial H_{\tau}}{\partial y}$,
        \item $\displaystyle \frac{\partial H_{\tau}}{\partial q}=\frac{\partial H_{\tau}}{\partial p}=0$.
    \end{enumerate}

    \noindent Note that (ii) implies (iv), therefore it remains to show that there exists $H_{\tau}$ which satisfies (i),(ii),(iii) and (v). Along the image of $\Gamma_{\tau}$ we have
    \begin{equation}\label{obstruction}
        (\cos m_k^2t-\sin m_k^2t)\Big(\frac{\partial}{\partial x}+y\frac{\partial}{\partial z}\Big)-(\sin m_k^2t+\cos m_k^2t)\frac{\partial}{\partial y}=\frac{m_k}{1+(1/w_0-1)\lambda_{\tau}}\frac{\partial\Gamma_{\tau}}{\partial w}
    \end{equation}
    \noindent Condition (i) implies that $dH_{\tau}(\partial\Gamma_{\tau}/\partial w)=0$, so applying $dH_{\tau}$ on the right-hand side of the above equation yields $0$. This is consistent with conditions (ii) and (iii) because
    \[
    (\cos m_k^2 t - \sin m_k^2 t) dy\left(\frac{\partial \Gamma_\tau}{\partial \tau}\right) 
    - (\sin m_k^2 t + \cos m_k^2 t) dx\left(\frac{\partial \Gamma_\tau}{\partial \tau}\right) = 0.
    \]

    \noindent The vectors 
    \[
    V_1 = \left(\frac{\partial}{\partial x} + y\frac{\partial}{\partial z}\right),\ 
    V_2 = \frac{\partial}{\partial y},\ 
    V_3 = \frac{\partial \Gamma_\tau}{\partial w},\ 
    V_4 = \frac{\partial \Gamma_\tau}{\partial t},\ 
    V_5 = \frac{\partial}{\partial q},\ 
    V_6 = \frac{\partial}{\partial p}
    \]
    \noindent span the tangent space $T_{\Gamma_\tau(t,s,w)}\mathbb{R}^5$. Therefore, beyond (\ref{obstruction}), there are no additional obstructions to defining $H_\tau$.\\

    \noindent For instance, we can define $H_\tau$ in a neighborhood of $\mathrm{Im}\,\Gamma_\tau$ using the exponential map:
    \[
        H_\tau\left(\mathrm{exp}_{\Gamma_\tau(t,s,w)}(a_1V_1 + \cdots + a_5V_5)\right) 
        = -a_1 \cdot dx\left(\frac{\partial \Gamma_\tau}{\partial \tau}\right) 
        - a_2 \cdot dy\left(\frac{\partial \Gamma_\tau}{\partial \tau}\right),
    \]
    \noindent where $a_1,\ldots,a_5$ are near $0$. One verifies directly that $H_\tau$ satisfies all required properties.
\end{proof}

\subsection*{Step 2:}

\noindent\textit{We start by picking an open neighbourhood $\mathcal{N}=Op(\mathrm{Im}\,\Pi)\subset\Psi'_k\circ\Phi_{k-1}(\mathcal{U}_{k+1})$. Then we choose $m_{k+1}$ large enough, so that we can find a Legendrian isotopy supported inside $\mathcal{N}$ which takes $\Lambda_{k+1}(t,s)=(\gamma_{m_{k+1}}(t),s,0)$ arbitrary $C^0$-close to $\Lambda_{k}$. Lastly, we extend this isotopy to get a $C^0$-small contactomorphism $\Psi''_k$ supported inside $\mathcal{N}$, and denote $\widetilde{\Lambda}_{k+1}=\Psi''_k\circ\Lambda_{k+1}$.}\\

\noindent Let $\mathcal{N}_1\subset\pi_{xyz}(\mathcal{V}_{k})$ and $\mathcal{N}_2\subset\pi_{qp}(\mathcal{V}_{k})$ be open subsets such that
\[\mathrm{Im}\,\Pi=\mathrm{Im}\,\Sigma\times[-1,1]\times\{0\}\subset\mathcal{N}:=\mathcal{N}_1\times\mathcal{N}_2\subset\Psi_k'\circ\Phi_{k-1}(\mathcal{U}_{k+1}).\]

\noindent For every $M\in\mathbb{N}$ let
\begin{equation}\label{EqForMathcal(T)_M}
    \mathcal{T}_{M}=\{t_1,\ldots,t_{|\mathcal{T}_M|}\}=\Big\{\frac{(8k+1)\pi}{4M^2}\mid k\in\mathbb{Z}\Big\}\cap(-1,1)
\end{equation}
be the subset where the front projection of $\gamma_M$ intersects the $z$-axis, such that $t_i<t_j$ for $i<j$, and there is a single zig-zag in the front projection of $\gamma_M$ on each interval $(t_i,t_{i+1})$.

\begin{claim}\label{Claim_Psi''_k}
    Fix $\varepsilon>0$. For $M\in\mathbb{N}$ large enough, there exists a contactomorphism $\Psi_k''$ supported inside $\mathcal{N}$ such that $d_{C^0}(\Psi''_k,\mathrm{Id})<C/2^k$ and $\Psi''_k(\gamma_M(t),s,0)=(\widetilde{\gamma}_M(t),s,0)$, where $\widetilde{\gamma}_{M}:[-1,1]\rightarrow\mathbb{R}^3$ is a curve that satisfies
    \begin{itemize}
        \item $d_{C^0}(\widetilde{\gamma}_{M},\gamma_{m_k})<\varepsilon$,
        \item $(\forall 1\leq i\leq |\mathcal{T}_M|-1)$ there is a single zig-zag in the front projection of $\psi^{\tau}\circ\gamma_M$ on the interval $(t_i,t_{i+1})$. 
    \end{itemize}
\end{claim}

\begin{proof}    
    For $\epsilon>0$ let $\lambda_{\tau}$ be a function defined by the Lemma~\ref{lambda_function_lemma}. Define an isotopy $\Theta_{\tau}:[-1,1]^2\rightarrow\mathbb{R}^5,\,\tau\in[0,1]$ as
    \[\Theta_{\tau}(t,s)=\Pi(t,s,\lambda_{\tau}(t,s))=(\Sigma(t,\lambda_{\tau}(t,s)),s,0).\]
    \noindent The isotopy $\Theta_{\tau}$ satisfies $\Theta_0=\Lambda_k$ and $\Theta_1=\Lambda_{\infty}$, and for small $\epsilon>0$ decomposes into two \textit{isotopies by parts} (Definition~\ref{isotopy_by_parts}) on $[0,1/2]$ and $[1/2,1]$. We will $C^0$-approximate $\Theta_{\tau}$ by a Legendrian isotopy $\Gamma_{\tau}$ (which is concatenation of two isotopies by parts as well) and conclude the proof via Lemma~\ref{IsotopyExtensionLemma} (applied twice).\\

    \noindent The strategy is to first approximate the family of curves $\Sigma(t,\lambda_{\tau}(t,s))$ by Legendrian curves, then lift them to a Legendrian embedding in $\mathbb{R}^5$ using: \textit{For a Legendrian family $\gamma_s(t)=(x_s(t),y_s(t),z_s(t))$ in $(\mathbb{R}^3,\ker(dz-ydx))$ where $y_s=\partial_t z_s/\partial_t x_s$, the map
    \[(t,s)\mapsto \left(\gamma_s(t),s,\alpha\left(\tfrac{d}{ds}\gamma_s(t)\right)\right)\]
    \noindent gives a Legendrian embedding in $(\mathbb{R}^5,\ker(dz-ydx-pdq))$.}\\

    \noindent Define the coordinate functions 
    \[
    (x_{\tau}(t,s), y_{\tau}(t,s), z_{\tau}(t,s)) := \Sigma(t, \lambda_{\tau}(t,s)).
    \]
    We need to construct a Legendrian approximation $\Gamma_{\tau} = (X_{\tau}, Y_{\tau}, Z_{\tau})$ satisfying:
    \begin{enumerate}
        \item The Legendrian condition $\partial_t Z_{\tau} = Y_{\tau}\partial_t X_{\tau}$,
        \item $d_{C^0}(\Gamma_{\tau}, (x_{\tau}, y_{\tau}, z_{\tau})) \leq \delta$,
        \item Isotopy-by-parts structure,
        \item The contact angle bound $|\alpha(\frac{\partial}{\partial s}\Gamma_{\tau})|=|\partial_s Z_{\tau} - Y_{\tau}\partial_s X_{\tau}| < \delta$.
    \end{enumerate}

    \noindent This construction yields, through the lifting procedure explained above, a Legendrian isotopy $\widetilde{\Lambda}_{\tau}$ supported in $Op(\mathrm{Im}\,\Pi)\subset\mathcal{N}$ when $\delta$ is small enough.\\

    \noindent To complete the construction, we need an isotopy $\Gamma_{\tau}(t,s)$ ($\tau\in[0,1]$) satisfying conditions $1$--$4$. The proof uses Legendrian approximation via front projection zig-zags (Theorem 2.44 in \cite{Ge08}): \textit{Decompose curve $t\mapsto (x(t),y(t),z(t))$ into segments $[t_i,t_{i+1}]$, then $C^0$-approximate the front projection $(x(t),z(t))$ on each segment by zig-zags with slopes near $y(t)$.}\\

    \noindent For each $i\in\{1,2,\ldots,\mathcal{T}_{M}\}$, define the base points and their offsets:
    \begin{equation}
        P_i = (x_\tau(t_i,s), z_\tau(t_i,s)), \quad 
        Q_i^\pm = P_i \pm d(1, y_\tau(t_i,s)).
    \end{equation}

    \noindent The zig-zag path $ P_i \to Q_i^+ \to Q_{i+1}^- \to P_{i+1}$ has segments with slopes:
    \begin{equation}
    \begin{aligned}
        \text{(1) } P_i \to Q_i^+ &: m_1 = y_i \\
        \text{(2) } Q_i^+ \to Q_{i+1}^- &: m_2 = \frac{\Delta z_i - d(y_i + y_{i+1})}{\Delta x_i - 2d} \\
        \text{(3) } Q_{i+1}^- \to P_{i+1} &: m_3 = y_{i+1}
    \end{aligned}
    \end{equation}
    where $\Delta x_i = x_{i+1}-x_i$, $\Delta z_i = z_{i+1}-z_i$, and $y_j \equiv y_\tau(t_j,s)$. As $M \to +\infty$, we have $\Delta x_i, \Delta y_i, \Delta z_i \to 0$ with $||P_iQ_i^{\pm}|| \to d$ and $||Q_i^+Q_{i+1}^-|| \to 2d$. For sufficiently small $d>0$ and large $M$, the piecewise linear zig-zag lies within $\mathcal{N}_1$, $C^0$-approximates the front projection $(x_\tau,y_\tau,z_\tau)$, and maintains slopes near $y_\tau(t,s)$. Finally, we want to modify linear paths between the vertices, so that the singularities at vertices turn into cusp singularities, yielding a Legendrian $\Gamma_\tau(t,s)$ satisfying properties $1$--$4$.\\

    \noindent Our aim is to construct a model for a curve interpolating between points $P,Q \in \mathbb{R}^2$ with target slopes $k_P$ and $k_Q$ at endpoints. In order to do so, we fix a model function $F(a,b,t)$ satisfying $F(a,b,0) = F(a,b,1) = 0$ with $\partial_t F(a,b,0) = a$ and $\partial_t F(a,b,1) = b$ (note that $F$ can be a polynomial).Then the interpolating curve is given by:
    \[
    \Gamma(P,Q,k_P,k_Q; t) = P + \|PQ\| \cdot R_\theta\left(t, F\left(\frac{k_P - \tan\theta}{1 + k_P\tan\theta}, \frac{k_Q - \tan\theta}{1 + k_Q\tan\theta}, t\right)\right)
    \]
    \noindent where $R_\theta$ denotes rotation by angle $\theta=\angle(Q-P)$. The curve satisfies $\Gamma(0) = P$ with slope $k_P$ and $\Gamma(1) = Q$ with slope $k_Q$. Finally, we define $\Gamma_{\tau}$ on the interval $[t_i,t_{i+1}]$ as
    \begin{equation}
    \begin{aligned}
        \text{(1) } P_i \to Q_i^+ &: (X_{\tau}(t,s),Z_{\tau}(t,s))=\Gamma(P_i,Q^+_i,m_1,m_1;v_1(t)) \\
        \text{(2) } Q_i^+ \to Q_{i+1}^- &: (X_{\tau}(t,s),Z_{\tau}(t,s))= \Gamma(Q^+_i,Q_{i+1}^-,m_1,m_2,v_2(t)) \\
        \text{(3) } Q_{i+1}^- \to P_{i+1} &: (X_{\tau}(t,s),Z_{\tau}(t,s))=\Gamma(Q_{i+1}^-,P_{i+1},m_2,m_2,v_3(t))
    \end{aligned}
    \end{equation}
    \noindent where $v_1,v_2,v_3$ are time reparametrizations, and $Y_{\tau}(t,s)$ satisfies the Legendrian condition. Note that the properties 2. and 3. are satisfied when and $d$ is small enough and $M\gg\frac{1}{d}$ is large enough. Moreover, for $t\in[t_i,t_{i+1}]$ we have
    \begin{equation*}
        \begin{aligned}
        \alpha\left(\frac{\partial}{\partial s}\Gamma_{\tau}(t,s)\right) 
        = \frac{\partial}{\partial s}Z_{\tau} - Y_{\tau}\frac{\partial}{\partial s}X_{\tau} 
        & = \frac{\partial}{\partial s}z_{\tau}(t_i,s) - y_{\tau}(t_i,s)\frac{\partial}{\partial s}x_{\tau}(t_i,s) + \mathcal{O}(d)\\
        & = \alpha\left(\frac{\partial}{\partial s}\Sigma(t,\lambda_{\tau}(t,s))\right)+\mathcal{O}(d) 
        = \mathcal{O}(d),
        \end{aligned}
    \end{equation*}
    \noindent so the property 4. follows provided $d$ is small enough.
\end{proof}
\noindent We end this step by defining $m_{k+1}$ large enough, so that the above claim is satisfied for $M=m_{k+1}$ and $\varepsilon=1/2^{k+1}$.

\subsection*{Step 3:}

\noindent\textit{This is the final, and most important step, where we construct a $C^0$-small contactomorphism between $\Lambda_k(t,s)=(\gamma_{m_k}(t),s,0)$ and $\widetilde{\Lambda}_{k+1}(t,s)=(\widetilde{\gamma}_{m_{k+1}}(t),s,0)$. Note that front projection of $\Lambda_{k}$ consists of many "equally-distributed" zig-zag strips, so the aim of this step is to find a $C^0$-small contactomorphism $\Psi$ which creates additional zig-zag strips in the front projection of $\Lambda_k$, so that $\Delta:=\Psi\circ\Lambda_k$ and $\widetilde{\Lambda}_{k+1}$ end up having the same number of zig-zag strips in the front projection.}\\

\noindent Using equation (\ref{CuspsSet}) we get
\[\#(\text{right cusps in the front projection of }\gamma_m)=1+\Big\lfloor\frac{m^2}{2\pi}-\frac{3}{8}\Big\rfloor+\Big\lfloor\frac{m^2}{2\pi}+\frac{3}{8}\Big\rfloor\]
\[\#(\text{left cusps in the front projection of }\gamma_m)=1+\Big\lfloor\frac{m^2}{2\pi}-\frac{1}{8}\Big\rfloor+\Big\lfloor\frac{m^2}{2\pi}+\frac{1}{8}\Big\rfloor.\]
\noindent Since the sequence $\frac{m^2}{2\pi}\,(\mathrm{mod} \,1)$ is equidistributed in $[0,1)$, we may assume that all $m_k$ that we choose satisfy $\frac{m_k^2}{2\pi}\,(\mathrm{mod} \,1)\in(0,1/8)$, and hence the number of left cusps and the right cusps of $\gamma_{m_k}$ is even and equal to $2\lfloor m^2/(2\pi)\rfloor$. Let 
\begin{equation}\label{Eq{T1,...,Tl}}
    \{T_1,T_2,\ldots,T_{l}\}\subset\mathcal{T}_{m_{k+1}}\text{ (see (\ref{EqForMathcal(T)_M}) for the definition of $\mathcal{T}_{m_{k+1}}$)}
\end{equation}

\noindent be the subset of points such that $T_1<T_2<\ldots <T_l$, and  $\gamma_{m_k}$ has either 4 or 6 cusps in the front projection on each interval in the collection
\begin{equation}\label{EqMathcal{I}Defn}
    \mathcal{I}:=\{(0,T_1),(T_1,T_2),\ldots,(T_l,1)\}
\end{equation}

\begin{figure}[h]
    \centering
    \includegraphics[scale=0.9]{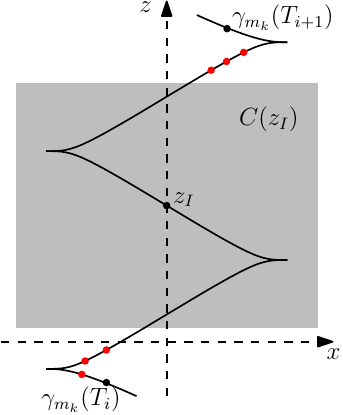}
    \caption{Front projection of $\gamma_{m_k}$; red dots represent $\mathrm{ZigZags}(I)$}
    \label{ZigZagSketch}
\end{figure}

\noindent For each interval $I=(T_i,T_{i+1})\in\mathcal{I}$ we pick a point $z_{I}\in I$ such that both intervals $(T_i,z_I)$ and $(z_I,T_{i+1})$ contain at least 2 cusps, and such that the front of $\gamma_{m_k}$ intersects the $z$-axis at the point $z_I$ (see \ref{EqZ-axisIntersection} for description of $z_I$). Moreover, denote by $C(z_I)\subset\mathbb{R}^3$ the open neighbourhood as in subsection \ref{SectionLooseCharts}, and denote

\begin{equation}\label{EqMathcal{C}Defn}
    \mathcal{C}:=\bigcup_{I\in\mathcal{I}}C(z_{I})
\end{equation}

\noindent Additionally, for each interval $I\in\mathcal{I}$ we pick a finite set of points $\mathrm{ZigZags}(I)\subset I$, such that $\gamma_{m_k}(\mathrm{ZigZags}(I))\cap C(z_I)=\emptyset$ and cardinality of the set $\mathrm{ZigZags}(I)$ is equal to the difference between number of right cusps of $\widetilde{\gamma}_{m_{k+1}}$ and $\gamma_{m_k}$ on the interval $I$, and denote

\begin{equation}\label{EqZigZags}
    \mathrm{ZigZags}:=\bigcup_{I\in\mathcal{I}}\mathrm{ZigZags}(I)
\end{equation}

\noindent For each $t\in\mathrm{ZigZags}$ let $\mathcal{U}_t=Op(\gamma_{m_k}(t))\subset\mathbb{R}^3$ and denote
\begin{equation}\label{EqMathcal{U}}
    \mathcal{U}=\bigsqcup_{t\in\mathrm{ZigZags}}\mathcal{U}_t
\end{equation}

\noindent Define a Legendrian curve
\[\chi:[-1,1]\rightarrow Op(\mathrm{Im}\,\gamma_{m_k})\subset\mathbb{R}^3\]
\noindent which has the same front projection as $\gamma_{m_k}$, except at the neighbourhood of points in $\mathrm{ZigZags}$, where for each $t\in\mathrm{ZigZags}$ $\chi$ has a zig-zag with image in $\mathcal{U}_t$. Therefore, Legendrian $\chi$ satisfies
\begin{itemize}
    \item $\chi$ and $\widetilde{\gamma}_{m_{k+1}}$ have the same number of left and right cusps (in the front projection) on each interval from the collection $\mathcal{I}$.
    \item the front projection of $\chi$ has no self-intersections.
\end{itemize}

\noindent Finally, define a Legendrian embedding $\Delta:[-1,1]^2\rightarrow\mathbb{R}^5$ as
\begin{equation}\label{EquationDelta}
    \Delta(t,s)=(\chi(t),s,0)
\end{equation}

\begin{claim}
    There exists a contactomorphism $\Psi_k'''$ compactly supported inside $\mathcal{N}$ such that $d_{C^0}(\Psi'''_k,\mathrm{Id})<C/2^k$ and $\Psi'''_k\circ\widetilde{\Lambda}_{k+1}=\Delta$.
\end{claim}

\begin{proof}
    Based on the properties of the curve $\chi$, we conclude that we can find a Legendrian isotopy $\Gamma_{\tau}:[-1,1]\rightarrow\mathbb{R}^3,\tau\in[0,1]$ between $\gamma_{m_{k+1}}$ and $\chi$ which additionally satisfies 
    \[\mathrm{supp}_{\tau\in[0,1]}\Gamma_{\tau}\subset\bigsqcup_{I\in\mathcal{I}}Op(\gamma_{m_k}(I))\subset\mathcal{N}_1.\]
    \noindent This isotopy translates zig-zags in the direction of $\frac{\partial}{\partial t}\gamma_{m_k}$. For every $I\in\mathcal{I}$, the diameter of the set $\gamma_{m_k}(I)$ is at most $30/m_k$ (follows from the fact that $\gamma_{m_k}(I)$ has at most 6 cusps in the front). We now finish the proof analogous to the proof of Claim \ref{Claim_Psi''_k}.
\end{proof}

\begin{prop}\label{PropZigZagStrips}
    There exists a contactomorphism $\Psi:\mathbb{R}^5\rightarrow\mathbb{R}^5$ supported in $\mathcal{N}$ with
    \begin{enumerate}
        \item $\Psi\circ\Lambda_{k}=\Delta$,
        \item $d_{C^0}(\Psi,\mathrm{Id})<100/m_k$.
    \end{enumerate}
\end{prop}

\noindent We finish the induction step by defining $\Psi_{k}:=(\Psi_k'')^{-1}\circ(\Psi_k''')^{-1}\circ\Psi\circ\Psi_k'$.

\section{Proof of Proposition \ref{PropZigZagStrips}}

\noindent We will get $\Psi$ by extending a Legendrian isotopy between $\Lambda_k$ and $\Delta$. However, we need to address the upper bound of the $C^0$-norm of $\Psi$, and this is achieved by extending special type of isotopy that we call \textit{isotopy by parts:}

\begin{defn}[isotopy by parts (IBP)]\label{isotopy_by_parts}
    An isotopy $\Gamma_t:[-1,1]^2\rightarrow\mathbb{R}^5,\,t\in[0,1]$ is called the isotopy by parts if it satisfies
    \begin{equation}\label{IsotopyByParts}
        \forall t\in[0,1/2]\,\, \mathrm{supp}(\Gamma_t)\subset\mathcal{A},\quad\forall t\in[1/2,1]\,\,\mathrm{supp}(\Gamma_t)\subset\mathcal{B}
    \end{equation}
    where $\mathcal{A}$ and $\mathcal{B}$ consist of finite number of connected disjoint open subsets of $\mathbb{R}^5$, and the diameter of each connected component of $\mathcal{A}\cup\mathcal{B}$ is less than $\frac{C}{2}\cdot d_{C^0}(\Gamma_0,\Gamma_1)$.
\end{defn}

\noindent If we find a Legendrian IBP between $\Lambda_k$ and $\Delta$, the Lemma \ref{IsotopyExtensionLemma} guarantees that we can extend this isotopy to get a contactomorphism $\Psi$ that satisfies $||\Psi||_{C^0}\leq C\cdot d_{C^0}(\Lambda_k,\Delta)$. The construction of the Legendrian IBP will be done indirectly in 3 steps.

\subsection*{Step I:}
\noindent\textit{We construct a \textbf{formal Legendrian IBP} between $\Lambda_k$ and $\Delta$ (see \ref{DefFormalLegIsotopy} for definition of formal Legendrian embeddings). Since we are working with Legendrian discs, it is easy to overcome topological obstructions of formal Legendrian isotopies, and we are essentially constructing a smooth IBP from $\Lambda_k$ to $\Delta$.}\\

\noindent Let $S_0,S_1\subset[-1,1]$ be two disjoint subsets such that
\begin{itemize}
    \item $[-1,1]\setminus S_0$ and $[-1,1]\setminus S_1$ are both disjoint unions of intervals of length $<10/m_k$,
    \item $[-1,1]\setminus(S_0\cup S_1)$ is disjoint union of intervals of length $>4/m_k$.
\end{itemize}
\noindent Fix $0<\varepsilon\ll1/m_k$ and define two subsets of $\mathbb{R}^5$, consisting of disjoint union of connected sets of small diameter
\begin{equation}\label{EquationMathcal(A,B)}
    \mathcal{A}=\mathcal{U}\times ([-1,1]\setminus S_0)\times (-\varepsilon,\varepsilon),\quad\mathcal{B}=\mathcal{U}\times ([-1,1]\setminus S_1)\times (-\varepsilon,\varepsilon).
\end{equation}

\begin{claim}
    There exists a \textbf{formal Legendrian isotopy} $\Gamma_{\tau}:[-1,1]^2\rightarrow\mathbb{R}^5,\tau\in[0,1]$, between $\Lambda_k$ and $\Delta$ which additionally satisfies
    \begin{equation}\label{splitingCondition}
        \mathrm{supp}_{\tau\in[0,1/2]}\Gamma_{\tau}\subset\mathcal{A},\quad\mathrm{supp}_{\tau\in[1/2,1]}\Gamma_{\tau}\subset\mathcal{B}
    \end{equation}
\end{claim}

\begin{proof}
    Recall that $\Lambda_k(t,s)=(\gamma_{m_k}(t),s,0)$ and $\Delta(t,s)=(\chi(t),s,0)$, where $\gamma_{m_k}$ and $\chi$ are Legendrian curves that coincide outside the set $[-1,1]\setminus Op(\mathrm{ZigZags})$, and for each $t_0\in\mathrm{ZigZags}$ the restriction $\chi|_{Op(t_0)}$ has a zig-zag in the front projection, while $\gamma_{m_k}|_{Op(t_0)}$ has a smooth front. Denote by 
    \[\widetilde{\gamma}_{\tau}:[-1,1]\rightarrow\mathbb{R}^3\]
    \noindent a homotopy of topological embeddings between $\gamma_{m_k}$ and $\chi$ with support in $\mathcal{U}$ (see (\ref{EqMathcal{U}})) that is singular in the neighbourhood of points $\tau=0,t\in\mathrm{ZigZags}$ and Legendrian otherwise, with the local model near the singularities
        \[(t,\tau)\mapsto\big(t^3-\tau\cdot t,\frac{15}{4}(t^2-\tau/3),\frac{9}{4}t^5-\frac{5\tau}{2}t^3+\frac{5\tau^2}{4}t\big),\]
    \begin{figure}[h]
    \centering
    \includegraphics[scale=1]{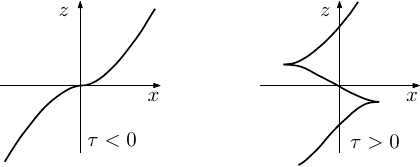}
    \caption{Front projection of $\widetilde{\gamma}_{\tau}$ ($\tau<0$ on the left, and $\tau>0$ on the right)}
    \label{ZigZagLocalModel}
    \end{figure}

    \noindent The next step is to extend $\widetilde{\gamma}_{\tau}$ to a smooth isotopy $\Gamma_{\tau}$ between $\Lambda_{m_k}$  and $\Delta$ (which can be further extended to a formal Legendrian isotopy). For this purpose we need some auxiliary functions. Let $\zeta:[-1,1]\rightarrow[1/4,3/4]$ be a smooth function which satisfies $\zeta|_{S_0}=1/4$ and $\zeta|_{S_1}=3/4$. Next, define a function $\lambda:[-1,1]\times[0,1]\rightarrow\mathbb{R}$ that satisfies $\lambda(s,0)=-1$, $\lambda(s,1)=1$, $\lambda(s,\zeta(s))=0$, $\lambda|_{S_0\times[0,1/2]}=-1$, $\lambda|_{S_1\times[1/2,1]}=1$ and $\frac{d}{d\tau}\lambda(s,\tau)>0$ (figure \ref{deltastau}).

\begin{figure}[h]
    \centering
    \includegraphics[width=0.9\textwidth]{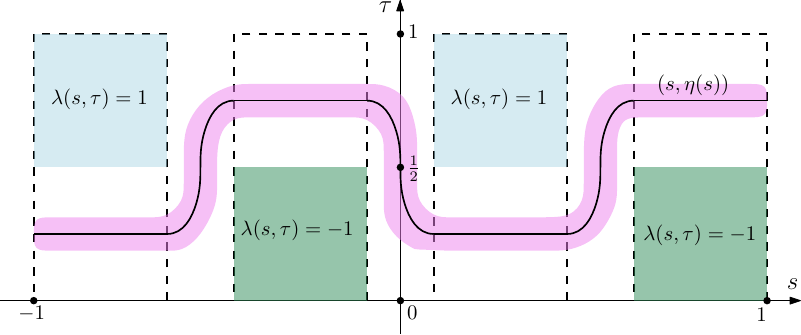}
    \caption{Properties of the function $\lambda(s,\tau)$. A neighbourhood of the zero set of $\lambda$ is coloured in violet.}
    \label{deltastau}
\end{figure}
    \noindent Additionally, define smooth function $h:\mathbb{R}^2\rightarrow [0,\varepsilon)$ supported inside $\mathcal{U}\times(-\varepsilon,\varepsilon)$ such that for each $t_0\in\mathrm{ZigZags}$ we have $h(x,y)=x-t_0$ for $(x,y)$ near $(t_0,0)$. One then easily verify that the smooth isotopy defined as
    \[\Gamma_{\tau}(t,s)=(\widetilde{\gamma}_{\lambda(s,\tau)}(t),s,h(t,\lambda(s,\tau)))\]
    \noindent satisfies all required properties, and can be extended to a formal Legendrian isotopy because the Reeb vector field $\partial/\partial z$ is transverse to $T\Gamma_{\tau}$, so we can project $T\Gamma_{\tau}$ to  $\xi$.
\end{proof}

\subsection*{Step II:}
\noindent We use \textit{relative $C^0$-dense approximation} (see Proposition \ref{WrinkledLegProp}) of the formal Legendrian isotopy $\Gamma_{\tau}$ from the previous step by a \textit{wrinkled Legendrian isotopy} $\Gamma'_{\tau}$ (see Definition \ref{DefWrinkledLeg}). The condition (\ref{splitingCondition}) and the fact that $\Gamma'_{\tau}$ is arbitrary $C^0$-close to $\Gamma_{\tau}$ implies that we can assume
\begin{equation}\label{MainRestriction}
    \mathrm{supp}_{\tau\in[0,1/2]}\Gamma'_{\tau}\subset\mathcal{A},\quad\mathrm{supp}_{\tau\in[1/2,1]}\Gamma'_{\tau}\subset\mathcal{B}
\end{equation}
\noindent Let $K\in\mathbb{Z}_+$ be the number of embryo singularities of the isotopy $\Gamma'_{\tau}$ where Legendrian wrinkles are created in the forward time. For each $1\leq i\leq K$ denote by $\tau^+_i$ the time when the $i$-th wrinkle was born, and by $\tau^-_i$ the time when the $i$-th wrinkle disappeared. Additionally, denote by $W^{\tau}_i$ (with $\tau\in(\tau^+_i,\tau^-_i)$) the $i$-th Legendrian wrinkle. The condition (\ref{MainRestriction}) implies
\begin{equation}\label{Restriction1}
    \begin{split}
        (\forall\tau\in(0,1/2])(\forall 1\leq i\leq K)\quad\tau_i^+<\tau<\tau_i^-\implies\Gamma'_{\tau}(W^{\tau}_i)\subset\mathcal{A},\\
        (\forall\tau\in[1/2,1))(\forall 1\leq i\leq K)\quad\tau_i^+<\tau<\tau_i^-\implies\Gamma'_{\tau}(W^{\tau}_i)\subset\mathcal{B}.
    \end{split}
\end{equation}

\subsection*{Step III:}

\noindent\textit{Lastly, we resolve Legendrian wrinkle singularities using \textit{twist markings} (see Definition \ref{DefTwistMarkings}) with the same techniques as Emmy Murphy used in \cite{Mu12} to prove that loose Legendrians satisfy $h$-principle. However, we choose twist markings more carefully, so that the resulting Legendrian isotopy satisfies condition (\ref{IsotopyByParts}).}\\

\noindent For each $I\in\mathcal{I}$ (see \ref{EqMathcal{I}Defn}) define
\[\widetilde{C}(I)=\{|x|<2/m,|y|<2/m,z\in I\}\subset\mathbb{R}^3,\]
\noindent and denote $\widetilde{\mathcal{C}}=\cup_{I\in\mathcal{I}}\widetilde{C}(I)$. Note that $C(z_I)\subset\widetilde{C}(I)$ (see \ref{EqMathcal{C}Defn}) and hence $\mathcal{C}\subset\widetilde{\mathcal{C}}$. Additionally, define
\begin{equation*}
    \widetilde{\mathcal{A}}=\widetilde{\mathcal{C}}\times([-1,1]\setminus S_0)\times(-3/m_k,3/m_k),\quad\widetilde{\mathcal{B}}=\widetilde{\mathcal{C}}\times([-1,1]\setminus S_1)\times(-3/m_k,3/m_k)
\end{equation*}
\noindent and note that $\mathcal{A}\subset\widetilde{\mathcal{A}},\mathcal{B}\subset\widetilde{\mathcal{B}}$ and the connected components of $\mathcal{A}$ are in "one to one" correspondence (as a subsets) with the connected components of $\widetilde{\mathcal{A}}$ (and same holds with the connected components of $\mathcal{B}$ and $\widetilde{\mathcal{B}}$).
\begin{claim}\label{ClaimPositionsOfLooseCharts}
    For each $1\leq i\leq K$ there exists an open subset 
    \[\mathcal{B}_i\subset(\widetilde{\mathcal{A}}\setminus\mathcal{A})\cap(\widetilde{\mathcal{B}}\setminus\mathcal{B})\]
    with the following properties
    \begin{enumerate}
        \item sets $\mathcal{B}_i$ are pairwise disjoint, and $\mathcal{B}_i\cap\mathrm{Im}\,\Lambda_{k}$ are loose charts,
        \item For $\tau\leq\frac{1}{2}$, $\Gamma'_{\tau}(W^{\tau}_i)$ and $\mathcal{B}_i$ are in the same connected component of the set $\widetilde{\mathcal{A}}$, and for $\tau\geq\frac{1}{2}$, $\Gamma'_{\tau}(W^{\tau}_i)$ and $\mathcal{B}_i$ are in the same connected component of the set $\widetilde{\mathcal{B}}$.
    \end{enumerate}
\end{claim}

\begin{proof}
    Based on property (\ref{Restriction1}) and the fact that $\tau\mapsto W^{\tau}_i$ is continuous, we conclude that every Legendrian wrinkle can visit at most one connected component of the set $\mathcal{A}$ when $\tau\leq 1/2$, and at most one connected component of the set $\mathcal{B}$ when $\tau\geq 1/2$. Therefore, to each Legendrian wrinkle we can assign a loose chart of the form
    \[C(z_I)\times J\times (-3/m_k,3/m_k)\,\,(\text{$J$ is a connected component of $[-1,1]\setminus(S_0\cup S_1)$, }\,I\in\mathcal{I}),\]
    \noindent so that condition 2. is satisfied (see (\ref{EqMathcal{I}Defn}) for definition of $\mathcal{I}$ and $z_I$, and Claim \ref{LooseChartsClaim} for description of loose chart). However, it might happen that we assign the same loose chart to a different Legendrian wrinkles, and this is resolved using the Proposition \ref{LooseChartsProp}.
\end{proof}

\noindent Define a new family $\Gamma''_{\tau}$ of wrinkled Legendrians that is equal to $\Gamma'_{\tau}$ outside $\sqcup_{i=1}^K \mathcal{B}_i$, and on each $\Lambda_{k}^{-1}(\mathcal{B}_i)=(\Gamma'_{\tau})^{-1}(\mathcal{B}_i)$ we replace the loose chart with an inside-out wrinkle (see (\ref{EqInside-outWrinkle})) with the same boundary conditions, constant in $\tau$.\\

\noindent For each $i\in\{1,2,\ldots,K\}$ and $\tau\in[0,1]$, we find a marking $\Phi^i_{\tau}\subset [-1,1]^2$ with the following properties
\begin{itemize}
    \item near the ends of the interval $[0,1]$, $\Phi^i_{\tau}$ is an interval contained in $(\Gamma''_{\tau})^{-1}(\mathcal{B}_i)$, exactly as defined for the model inside-out wrinkle,
    \item $\Phi^i_{\tau}$ is either diffeomorphic to an interval $[0,1]$ or a union of 2 intervals $S^0\times[0,1]$,
    \item $\Phi^i_{\tau^+_i}$ and $\Phi^i_{\tau^-_i}$ contain $i$-th embryo singularities,
    \item other than the component in $(\Gamma''_{\tau})^{-1}(\mathcal{B}_i)$, the boundary of $\Phi^i_{\tau}$ is exactly the 0-sphere of Legendrian wrinkles $W^{\tau}_i$,
    \item $(\forall\tau\in[0,1/2])\,\Gamma''_{\tau}(\Phi^i_{\tau})\subset\widetilde{\mathcal{A}}$ and $(\forall\tau\in[1/2,1])\,\Gamma''_{\tau}(\Phi^i_{\tau})\subset\widetilde{\mathcal{B}}$.
\end{itemize}

\noindent The markings which satisfy first four out of five listed properties can always be found (those are exactly the same properties as in the proof of Theorem 1.3. in \cite{Mu12}), and the last property follow from the property 2. in Claim \ref{ClaimPositionsOfLooseCharts}. We now apply Lemma \ref{LemmaResolvingSingularities} for each $\Phi^i_{\tau}$, one at a time. The resulting family
\[\widetilde{\Gamma''}_{\tau}:[-1,1]^2\rightarrow(\mathbb{R}^5,dz-ydx-pdq),\,\tau\in[0,1]\]
\noindent is a family of smooth Legendrians which furthermore satisefies:
\begin{enumerate}
    \item $(\forall\tau\in[0,1/2])\,\widetilde{\Gamma''}_{\tau}(\Phi^i_{\tau})\subset\widetilde{\mathcal{A}}$ and $(\forall\tau\in[1/2,1])\,\widetilde{\Gamma''}_{\tau}(\Phi^i_{\tau})\subset\widetilde{\mathcal{B}}$,
    \item $(\forall\tau\in\{0,1\})\,\widetilde{\Gamma''}_{\tau}$ is isotopic to $\Gamma_\tau$, via an isotopy supported in $\sqcup_{i=1}^K\mathcal{B}_i$.
\end{enumerate}

\noindent The above properties imply we can construct a Legendrian isotopy $\Gamma'''_{\tau}:[0,1]\rightarrow\mathbb{R}^5$ such that $\Gamma'''_0=\Gamma_0=\Lambda_{k}$, $\Gamma'''_1(t,s)=\Gamma_1=\Delta$ and 
\[(\forall\tau\in[0,1/2])\,\Gamma'''_{\tau}(\Phi^i_{\tau})\subset\widetilde{\mathcal{A}},\quad(\forall\tau\in[1/2,1])\,\Gamma'''_{\tau}(\Phi^i_{\tau})\subset\widetilde{\mathcal{B}}.\]
\noindent Finally, we get a contactomorphism $\Psi$ using the Lemma \ref{IsotopyExtensionLemma} that satisfies $d_{C^0}(\Psi,\mathrm{Id})<100/m_k$.

\section{Technical Lemmas}

\noindent The majority of contact diffeomorphisms we will construct by extending Legendrian isotopies. Additionally, we want to control the $C^0$-norm of the extension in terms of \textit{$C^0$-size} (defined below) of the isotopy.

\begin{defn}[support on the time interval]
    Let $\Gamma_{\tau}:X\rightarrow Y,\,\tau\in I$ be a smooth isotopy. The support of $\Gamma_{\tau}$ at time $\tau_0\in I$ is defined as
    \[\mathrm{supp}\,\Gamma_{\tau_0}:=\mathrm{Closure}\{\Gamma_{\tau_0}(x)\mid\frac{d}{d\tau}|_{\tau=\tau_0}\Gamma_{\tau}(x)\neq 0\}\subset Y.\] The support of the isotopy $\Gamma_{\tau}$ on the interval $J\subset I$ is defined to be
    \[\mathrm{supp}_{\tau\in J}\Gamma_{\tau}:=\bigcup_{\tau\in J}\mathrm{supp}\,\Gamma_{\tau}.\]

\end{defn}

\begin{defn}[size of the isotopy]
    Let $\Gamma_{\tau}:X\rightarrow Y,\,\tau\in I$ be an isotopy, and let $d$ be a Riemannian distance on $Y$. We define the size of the isotopy $\Gamma_{\tau}$ on the interval $J\subset I$ as the supremum of the diameters of the connected components of $\mathrm{supp}_{\tau\in J}\Gamma_{\tau}$
    \[\mathrm{size}_{\tau\in J}\,\Gamma_{\tau}:=\sup\{\mathrm{diam}(S)\mid S\subset\mathrm{supp}_{\tau\in J}\Gamma_{\tau}\text{ and }S\text{ is connected}\}.\]
\end{defn}

\noindent The next lemma is the main $C^0$-estimate for a contact diffeomorphism that we get by extending a Legendrian isotopy.

\begin{lemma}\label{IsotopyExtensionLemma}
    Let $\Gamma_{\tau}:\Lambda\rightarrow (Y,\xi),\tau\in[0,1]$ be an isotopy of Legendrian embeddings. For every $\varepsilon>0$ there exists a contact isotopy $\phi^{\tau}:(Y,\xi)\rightarrow(Y,\xi)$ such that $\phi^{\tau}\circ\Gamma_0=\Gamma_{\tau}$ and
    \[\max_{\tau\in[0,1]}\,d_{C^0}(\phi^{\tau},\mathrm{Id})<\mathrm{size}_{\tau\in[0,1/2]}\Gamma_{\tau}+\mathrm{size}_{\tau\in[1/2,1]}\Gamma_{\tau}+\varepsilon.\] 
\end{lemma}

\begin{proof}
    Let $\mathrm{supp}_{\tau\in[0,1/2]}\Gamma_{\tau}=\bigsqcup_{i=1}^k A_i\subset Y$ and $\mathrm{supp}_{\tau\in[1/2,1]}\Gamma_{\tau}=\bigsqcup_{j=1}^l B_j\subset Y$, where $A_1,\ldots,A_k,B_1,\ldots,B_l$ are connected closed subsets of $Y$. For each $i\in\{1,\ldots,k\}$ let $\mathcal{A}_i$ be an open subset of $Y$, and for each $j\in\{1,\ldots,l\}$ let $\mathcal{B}_j$ be an open subset of $Y$ such that the following holds
    \begin{itemize}
        \item $A_i\subset\mathcal{A}_i$, and $B_j\subset\mathcal{B}_j$,
        \item $\mathrm{diam}(\mathcal{A}_i)<\mathrm{diam}(A_i)+\varepsilon/2$ and $\mathrm{diam}(\mathcal{B}_j)<\mathrm{diam}(B_j)+\varepsilon/2$,
        \item $\mathcal{A}_1,\ldots,\mathcal{A}_k$ are pairwise disjoint, and $\mathcal{B}_1,\ldots,\mathcal{B}_l$ are pairwise disjoint.
    \end{itemize}
    \noindent Denote $\mathcal{A}=\bigcup_{i=1}^k\mathcal{A}_i$ and $\mathcal{B}=\bigcup_{j=1}^l\mathcal{B}_j$. The Legendrian isotopy $\Gamma_{\tau}$ can be extended to an ambient contact isotopy $\phi^{\tau}_{H}$, where $H_{\tau}:Y\rightarrow Y$ is a time-dependent contact Hamiltonian which is equal to $0$ outside an arbitrary small tubular neighbourhood of $\mathrm{supp}\,\Gamma_{\tau}$ (see Theorem 2.41. in \cite{Ge08}). Thus, we can ensure that $H_{\tau}$ satisfies
    \[H_{\tau}(x)=0 \text{ for }(x,\tau)\in (Y\setminus\mathcal{A})\times[0,1/2]\cup(Y\setminus\mathcal{B})\times[1/2,1].\]
    \noindent It only remains to verify that $\phi^{\tau}:=\phi^{\tau}_{H}$ satisfies the required $C^0$ bound, so we consider the following cases:
    \begin{itemize}
        \item $x\in Y\setminus(\mathcal{A}\cup\mathcal{B})$, then $d(\phi^{\tau}(x),x)=0$ for all $\tau\in[0,1]$,
        \item $x\in\mathcal{B}\setminus\mathcal{A}$, then there exists $j\in\{1,\ldots,l\}$ such that $\phi^{\tau}(x)\in\mathcal{B}_j$ for all $\tau\in[0,1]$. Therefore we have $d(\phi^{\tau}(x),x)\leq\mathrm{diam}(\mathcal{B}_j)<\mathrm{diam}(B_j)+\varepsilon/2\leq\mathrm{size}_{\tau\in[1/2,1]}\Gamma_{\tau}+\varepsilon/2$, 
        \item $x\in\mathcal{A}$, then there exists $i\in\{1,\ldots,k\}$ such that $\phi^{\tau}(x)\in\mathcal{A}_i$ for all $\tau\in[0,1/2]$, thus we have 
        \[(\forall\tau\in[0,1/2])\,d(\phi^{\tau}(x),x)\leq\mathrm{diam}(\mathcal{B}_j)<\mathrm{diam}(B_j)+\varepsilon/2\leq\mathrm{size}_{\tau\in[1/2,1]}\Gamma_{\tau}+\varepsilon/2\]
        For $\tau\in[1/2,1]$, we further distinguish two cases:
        \begin{itemize}
            \item $\phi^{1/2}(x)\in\mathcal{B}$, then there exists $j\in\{1,\ldots,l\}$ such that $\phi^{\tau}(x)\in\mathcal{B}_j$ for all $\tau\in[1/2,1]$. Therefore we have
            \[d(\phi^{\tau}(x),\phi^{1/2}(x))\leq\mathrm{diam}(\mathcal{B}_j)<\mathrm{diam}(B_j)+\varepsilon/2\leq\mathrm{size}_{\tau\in[1/2,1]}\Gamma_{\tau}+\varepsilon/2.\]
            Finally, the triangle inequality gives
            \[d(\phi^{\tau}(x),x)\leq d(\phi^{\tau}(x),\phi^{1/2}(x))+d(\phi^{1/2}(x),x)<\mathrm{size}_{\tau\in[0,1/2]}\Gamma_{\tau}+\mathrm{size}_{\tau\in[1/2,1]}\Gamma_{\tau}+\varepsilon.\]
            \item $\phi^{1/2}(x)\not\in\mathcal{B}$, then $\phi^{\tau}(x)=\phi^{1/2}(x)$ for all $\tau\in[1/2,1]$, therefore we have $d(\phi^{\tau}(x),x)\leq\mathrm{diam}(\mathcal{A}_i)<\mathrm{diam}(A_i)+\varepsilon/2\leq\mathrm{size}_{\tau\in[1/2,1]}\Gamma_{\tau}+\varepsilon/2$.
        \end{itemize}
    \end{itemize}
\end{proof}

\begin{lemma}\label{lambda_function_lemma}
For any $\epsilon > 0$, there exists a smooth map $\lambda_\tau : [-1,1]^2 \to [0,1]$ for $\tau \in [0,1]$ satisfying:
\begin{enumerate}[label=(\roman*)]
    \item $\lambda_0 \equiv 0$, $\lambda_1 \equiv 1$ and $\frac{\partial \lambda_\tau}{\partial \tau} \geq 0$ (non-decreasing),
    \item For each $i \in \{0,1,2,3\}$ and $\tau \in [i/4,(i+1)/4]$, the support $\mathrm{supp}\,\lambda_{\tau}\subset\mathcal{A}_i$ where $\mathcal{A}_i\subset[-1,1]^2$ is a disjoint union of open sets of diameter $<\epsilon$.
\end{enumerate}
\end{lemma}

\begin{proof}
    Choose subsets $A,B \subset [-1,1]$ such that $A \cap B = \emptyset$, and all sets $A,B,[-1,1]\setminus A$ and $[-1,1]\setminus B$ are unions of intervals of length $< \epsilon/\sqrt{2}$. Define $\lambda_\tau$ for $\tau\in [0,1]$ via:
    \begin{itemize}
        \item $(\forall\tau\in[0,1/4])$ $\frac{d}{d\tau}\lambda_{\tau}|_{A\times[-1,1]\cup [-1,1]\times A}=0$ and $\lambda_{1/4}|_{B\times B}=1$,
        \item $(\forall\tau\in[1/4,1/2])$ $\frac{d}{d\tau}\lambda_{\tau}|_{A\times[-1,1]\cup [-1,1]\times B}=0$ and $\lambda_{1/2}|_{B\times[-1,1]}=1$,
        \item $(\forall\tau\in[1/2,3/4])$ $\frac{d}{d\tau}\lambda_{\tau}|_{B\times[-1,1]\cup [-1,1]\times A}=0$ and $\lambda_{3/4}|_{B\times[-1,1]\cup [-1,1]\times B}=1$,
        \item $(\forall\tau\in[3/4,1])$ $\frac{d}{d\tau}\lambda_{\tau}|_{B\times[-1,1]\cup [-1,1]\times B}=0$ and $\lambda_1\equiv 1$.
    \end{itemize}
    \noindent The support of $\lambda_{\tau}$ for $\tau\in[i/4,(i+1)/4]$ lies inside the union of squares of sides $\epsilon/\sqrt{2}$.
\end{proof}

\subsection{Loose charts on $\Lambda_k$}\label{SectionLooseCharts}

\noindent In the front projection $\mathbb{R}^2(x,z)$, the curve $\gamma_{m_k}$ (see (\ref{Definition_gamma_m(t)}) for the explicit description) has cusps for
\begin{equation}\label{CuspsSet}
    t\in\mathrm{Cusps}(\gamma_{m_k}):=\Big\{\frac{(4i-1)\pi}{4m_k^2}\,\big\vert\, i\in\mathbb{Z}\Big\}\cap(-1,1)
\end{equation}
\noindent Coordinates of the cusps of $\gamma_{m_k}$ in the front projection have form
\[Q_i=\Big(\frac{-\sqrt{2}\cos(i\pi)}{m},\frac{(4i-1)\pi}{4m_k^2}\Big)\subset\mathbb{R}^2(x,z).\]
\noindent The front projection of the curve $\gamma_{m_k}$ intersects the $z$-axis at points
\begin{equation}\label{EqZ-axisIntersection}
    z_i=\pi_{z}\circ\gamma_m\Big(\frac{(4i+1)\pi}{4m_k^2}\Big)=\frac{(4i+1)\pi}{4m_k^2}
\end{equation}

\begin{figure}[h]
    \centering
    \includegraphics[scale=0.7]{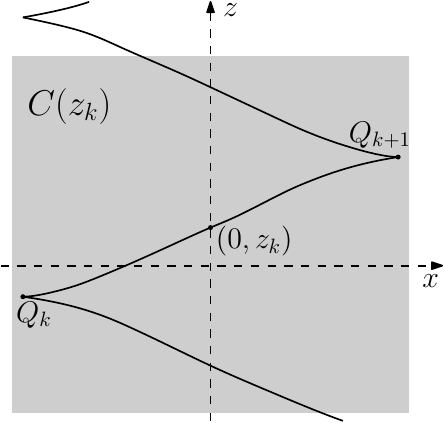}
    \caption{Front projection of $\gamma_{m_k}$ (gray area represents a projection of a loose chart)}
    \label{ZigZagFigure}
\end{figure}

\noindent For each $z_i$ defined as above, $\rho>1$ and $q_0\in(-1,1)$ define
\begin{equation*}
    \begin{split}
        C(z_i) & =\big\{|x|<\frac{3}{2m_k}, |y|<\frac{3}{2m_k}, |z-z_i|< \frac{9}{2m_k^2}\big\}\subset\mathbb{R}^3(x,y,z), \\
        V_{\rho}(q_0) & =\big\{|q-q_0|<\frac{3\rho}{2m_k}, |p|< \frac{3\rho}{2m_k}\big\}\subset\mathbb{R}^2(q,p), \\
        Z_{\rho}(q_0 & )=\big\{(q,0)\in\mathbb{R}^2\mid |q-q_0|< \frac{3\rho}{2m_k}\big\}\subset\mathbb{R}^2(q,p).
    \end{split}
\end{equation*}

\begin{claim}\label{LooseChartsClaim}
    The pair $(C(z_i)\times V_{\rho}(q_0),(C(z_i)\times V_{\rho}(q_0))\cap\, \mathrm{Im}\,\Lambda_k)$ is a loose chart.
\end{claim}

\begin{proof}
    A Legendrian isotopy of a Legendrian arc $\gamma_{m_k}|_{C(z_i)}$ (where $C(z_i)\subset(\mathbb{R}^3,dz-ydx)$ is a contact submanifold) inside $C(z_i)$ lifts to a Legendrian isotopy of $\Lambda_k|_{C(z_i)\times V_{\rho}(q_0)}$ inside $C(z_i)\times V_{\rho}(q_0)$ (we simply keep $q$ and $p$ coordinates fixed). We find a Legendrian isotopy (supported inside $C(z_i)$) of arc $\gamma_{m_k}|_{C(z_i)}$ which brings its ends to $\{y=z=0\}$ near the boundary. As we explained above, this isotopy lifts to a Legendrian isotopy of $\Lambda_k|_{C(z_i)\times V_{\rho}(q_0)}$ which can be realized as a contact isotopy $\Psi_t:C(z_i)\times V_{\rho}(q_0)\rightarrow C(z_i)\times V_{\rho}(q_0)$.\\
    
    \noindent Define a contact diffeomorphism $\Phi: C\times V_{\rho}\rightarrow [-1/2,1/2]^3\times[-\rho/2,\rho/2]^2$ as
\[\Phi(x,y,z,q,p)=\Big(\frac{m_k}{3}x,\frac{m_k}{3}y,\frac{m_k^2}{9}(z-z_i),\frac{m_k}{3}(q-q_0),\frac{m_k}{3}p\Big),\]
\noindent and note that
\[\Phi\circ\Psi_1((C\times V_{\rho})\cap\mathrm{Im}\Lambda_k)=\{\text{standard zig-zag in }[-1/2,1/2]^3\}\times[-\rho/2,\rho/2]\times\{0\}.\]
\end{proof}

\appendix
\section{Appendix}
\addtocontents{toc}{\protect\setcounter{tocdepth}{0}}
\setcounter{thm}{0}
\renewcommand{\thethm}{\Alph{section}.\arabic{thm}}

\subsection{Formal Legendrian embeddings}

\noindent Let $(Y,\xi)$ be a contact manifold of dimension $2n+1$, and let $\Lambda$ be a $n$-dimensional manifold. A \textit{Legendrian embedding} is a smooth embedding $f:\Lambda\rightarrow Y$ so that $df(T\Lambda)\subset\xi$. The simplest possible invariant one can develop beyond smooth isotopy type is the \textit{formal Legendrian} isotopy type.

\begin{defn}\label{DefFormalLegIsotopy}
    A formal Legendrian embedding is a pair $(f,F_s)$, where $f :\Lambda\rightarrow Y$ is a smooth embedding, and $F_s : T\Lambda\rightarrow TY$ is a homotopy of bundle maps covering $f$, so that:
    \begin{itemize}
        \item $F_0=df$,
        \item $F_s$ is fiberwise injective for all $s\in[0,1]$, and
        \item the image of $F_1$ is contained in $\xi$, and furthermore its image is Lagrangian with respect to the linear conformal symplectic structure on $\xi$.
    \end{itemize}
    Two Legendrian embeddings $f_0, f_1 : \Lambda\rightarrow (Y, \xi)$ are formally isotopic if there is a path of formal Legendrian embeddings interpolating between them.
\end{defn}

\noindent We note that every Legendrian embedding is a formal Legendrian embedding, by letting $F_s = df$ for all $s$. The advantage of working with formal Legendrians is that they are purely algebro-geometric objects. In particular, if $\Lambda$ is contractible, then any smooth isotopy of $\Lambda$ is a formal Legendrian isotopy.

\subsection{Wrinkled embeddings and wrinkled Legendrians}

\noindent We start by describing the model for singularities needed to define wrinkled embeddings. For $\delta\in\mathbb{R}$, define a plane curve $\psi_{\delta}:\mathbb{R}\rightarrow\mathbb{R}^2$
\[\psi_{\delta}(u)=\Big(u^3-\delta u,\frac{9}{4}u^5-\frac{5\delta}{2}u^3+\frac{5\delta^2}{4}u\Big).\]

\noindent We will assume that each $\psi_{\delta}$ is a compactly supported curve, which is false as written but it is easily accomplished with a cut-off function.\\

\noindent A \textit{wrinkle} is a map $w:\mathbb{R}^n\rightarrow\mathbb{R}^{n+1}$ given by
\[w(x,u)=(x,\psi_{1-||x||^2}(u)),\quad (x,u)\in\mathbb{R}^{n-1}\times\mathbb{R}.\]
\noindent A wrinkle is singular on the sphere $\{||x||^2+3u^2=1\}$, and we distinguish \textit{cusp singularities} on the upper and lower hemispheres, as well as \textit{unfurled swallowtail} singularities along the equator $\{u=0\}$. In parametric families of wrinkles $w_t:\mathbb{R}^n\rightarrow\mathbb{R}^{n+1}$ given by
\[w_t(x,u)=(x,\psi_{t-||x||^2}(u))\]
\noindent we define an \textit{embryo} singularity to be the singularity of $w_t$ at $t=x=u=0$. As we can see, an embryo singularity is a singularity allowing wrinkles to appear or disappear in parametric families of maps.

\begin{defn}
    Let $V$ and $W$ be manifolds, with $\mathrm{dim}\,W - 1=\mathrm{dim}\,V=n$. A \textit{wrinkled embedding} is a smooth map $f:V\rightarrow W$, which is a topological embedding, and which is singular on some finite collection of codimension $1$ spheres $S^{n-1}_j\subset V$, which bound disks $D^n_j\subset V$. Near each $S^{n-1}_j$ the map $f$ is required to be modeled on a wrinkle. In parametric families, wrinkled embeddings are allowed to have embryo singularities.
\end{defn}

\noindent We are now ready to give the definition of wrinkled Legendrian embeddings.

\begin{defn}\label{DefWrinkledLeg}
    Let $\Lambda$ be a smooth $n$-manifold and $(Y,\xi)$ a contact $(2n+1)$-manifold. A \textit{wrinkled Legendrian embedding} is a smooth map $f:\Lambda\rightarrow Y$, which is a topological embedding, satisfying the following properties
    \begin{itemize}
        \item the image of $df$ is contained in $\xi$ everywhere, and $df$ is full rank outside a subset of codimension $2$,
        \item the singular set is required to be diffeomorphic to a disjoint union of $(n-2)$-spheres $\{S^{n-2}_j\}$, called \textit{Legendrian wrinkles},
        \item Each $S^{n-2}_j$ is contained in a Darboux chart $U_j$, so that $\Lambda\cap U_j$ is diffeomorphic to $\mathbb{R}^n$, and the front projection $\pi_j\circ f:\Lambda\cap U_j\rightarrow\mathbb{R}^{n+1}$ of $f$ is a wrinkled embedding.
    \end{itemize}
    In particular, the front projection of each $S^{n-2}_j$ is the unfurled swallowtail singular set of a single wrinkle in the front. For parametric families of wrinkled Legendrians we also allow \textit{Legendrian embryos}; Legendrian lifts from the front projection of embryo singularities.
\end{defn}

\begin{rem}
    We point out that the collection of Darboux charts $\{U_j\}$ is considered part of the data of a wrinkled Legendrian.
\end{rem}

\begin{prop}[See Proposition 3.4. in \cite{Mu12}]\label{WrinkledLegProp}
    Let $(f_t,F_{s,t})$ be a parametric family of formal Legendrian embeddings $\Lambda\rightarrow(Y,\xi)$, $t\in I$. Then the family $(f_t,F_{s,t})$ is homotopic through formal Legendrian embeddings to a family $\overline{f}_t:\Lambda\rightarrow(Y,\xi)$ of wrinkled Legendrian embeddings. If $(f_t,F_{s,t})$ is already a wrinkled Legendrian embedding on a closed subset $A\subset\Lambda\times I$, then we can take $\overline{f}_t=f_t$ on this set. Moreover, the statement is $C^0$-dense.
\end{prop}
\begin{rem}
    If $\Lambda$ is contractible, any smooth isotopy is formal Legendrian isotopy, hence the theorem holds if $f_t$ is a smooth isotopy.
\end{rem}

\subsection{Twist markings and loose Legendrians}

\begin{defn}\label{DefTwistMarkings}
    Let $f:\Lambda\rightarrow (Y, \xi)$ be a wrinkled Legendrian with $k$ wrinkles, and let $\Phi\subset\Lambda$ be an embedded codimension $1$ smooth compact submanifold with boundary. Assume $\Phi$ has the topology of a sphere with $k$ open disks removed. Then $\Phi$ is called a \textit{twist marking} if the singular set of $\Lambda$ is equal to $\partial\Phi$, and there is a small neighborhood of the singular set so that $\Phi=\{u=0,||x||\geq 1\}\subset\Lambda$ in the local model $(\pi_j\circ f)(\Lambda\cap U_j)\cong w(\mathbb{R}^n)$.
    
    In parametric families $f_t : \Lambda\rightarrow (Y,\xi)$, we require the family $\Phi_t \subset\Lambda$ to be smoothly varying in $t$ whenever we are disjoint from the set of embryo singularities. At an embryo singularity modeled in the front projection $(\pi_j\circ f_t)(\Lambda\cap U_j)\cong w_t(\mathbb{R}^n)$, we require $\Phi_t = \{u = 0,||x||^2 \geq t\}$.
\end{defn}

\noindent The idea of a marking is that it is a pattern that gives us a canonical way to desingularize wrinkles.

\begin{lemma}[Lemma 4.2. in \cite{Mu12}]\label{LemmaResolvingSingularities}
    Let $f_t:\Lambda\rightarrow(Y,\xi)$ be a family of wrinkled Legendrians with $t\in D^k$, and let $\Phi_t\subset\Lambda$ be a family of markings for $f_t$. Then $f_t$ is $C^0$-close to a wrinkled Legendrian family $\widetilde{f}_t:\Lambda\rightarrow (Y,\xi)$, so that $\widetilde{f}_t$ is smooth on a neighborhood of $\Phi_t$ and equal to $f_t$ outside that neighborhood.
\end{lemma}

\noindent So far, if we want to construct a smooth family of Legendrians, it is enough to find a marking. This motivates the definition of loose Legendrians: intuitively they are smooth Legendrians which look like resolutions of some wrinkled Legendrian along some marking.\\

\noindent Let $C\subset\mathbb{R}^3$ be the cube of side length $1$, and let $\Lambda_0\subset C$ be a properly embedded Legendrian arc whose front is a zig-zag and which is equal to the set $\{y = z = 0\}$ near the boundary. Define open sets
\[V_{\rho}=\{|q|<\rho,|p|<\rho\}\subset T^*\mathbb{R}^{n-1}\cong\mathbb{R}^{n-1}\times\mathbb{R}^{n-1},\quad Z_{\rho}=V_{\rho}\cap\{p=0\}.\]
\noindent The set $\Lambda_0\times Z_{\rho}$ is a Legendrian submanifold of the open subset $C\times V_{\rho}\subset\mathbb{R}^{2n+1}$, with respect to the standard contact form.

\begin{defn}
    Let $\Lambda\subset(Y,\xi)$ be a Legendrian submanifold, and let $\rho>1$. We say that $\Lambda$ is \textit{loose} if there is an open subset $U\subset Y$, so that $(U,U\cap \Lambda)$ is contactomorphic as a pair to $(C\times V_{\rho},\Lambda_0\times Z_{\rho})$. The pair $(U,U\cap\Lambda)$ is then called a \textit{loose chart}.
\end{defn}

\begin{prop}[See Proposition 4.4. in \cite{Mu12}]\label{LooseChartsProp}
        For any $\overline{\rho}>1$, an arbitrary loose chart contains another loose chart of size parameter $\overline{\rho}$. In particular a loose chart contains two disjointly embeddded loose charts, and therefore a loose chart contains infinitely many disjointly embedded loose charts.
\end{prop}

\noindent Define an \textit{inside-out wrinkle} to be the map $\overline{w}:\mathbb{R}^n\rightarrow\mathbb{R}^{n+1}$ defined by
\begin{equation}\label{EqInside-outWrinkle}
    \overline{w}(x,u)=(x,\psi_{||x||^2-1}(u))
\end{equation}
\noindent Then $\overline{w}$ is singular on the hyperbola $\{||x||^2-3u^2=1\}$, which has unfurled swallowtails on the subset $\{u=1-||x||^2=0\}$ and cusps elsewhere. The wrinkled Legendrian $f:\mathbb{R}^n\rightarrow\mathcal{J}^1\mathbb{R}^n$ whose front projection is $\overline{w}$ is smooth outside of a compact set, but it is not standard. Let $\Phi=\{u=0,||x||\leq 1\}$. Then $\Phi$ is a marking for $f$. Moreover, if we resolve $f$ along $\Phi$ to obtain $\widetilde{f}$, then $\widetilde{f}(\mathbb{R}^n)\cap B^{2n+1}(\rho)$ is contactomorphic to a loose chart for any $\rho>1$. (see section 4.2. in \cite{Mu12} for more details).

\bibliographystyle{apacite}

\end{document}